\documentclass[12pt]{article}

\usepackage{amssymb}
\usepackage{amsthm}
\usepackage{amsmath}

\newtheorem{theorem}{Theorem}[section]
\newtheorem{lemma}[theorem]{Lemma}

\newtheorem{corollary}[theorem]{Corollary}

\begin{document}
\title{\bf On the minimum order \\ of a quadrangulation \\ on a given closed 2-manifold}
\author{\Large \bf Serge Lawrencenko \medskip\\ \small Department of Higher Mathematics 1,
National Research\\ \small University of Electronic Technology,
Zelenograd, Russia\\ \small E-mail: lawrencenko@hotmail.com}
\renewcommand{\today}{\small August 26, 2012}
\maketitle

\begin{abstract}
A partial formula is provided to calculate the smallest number of vertices possible in a quadrangulation on the closed orientable 2-manifold of given genus. This  extends the previously known partial formula due to N. Hartsfield and G. Ringel [J. Comb. Theory, Ser. B, 1989, 46, 84-95]. 
\end{abstract}

\smallskip


\section{Introduction}
\label{1}

A polyhedron $G \hookrightarrow \Sigma_g $ with 1-skeleton $G$ on the closed orientable 2-manifold $\Sigma_g$ of genus $g$  is called a {\it quadrangulation} if each face of the polyhedron is bounded by a cycle of four edges in $G$. The 1-skeleton $G$ is called the graph of the quadrangulation. A quadrangulation on $\Sigma_g$ is said to be {\it minimal} for $\Sigma_g$ if no quadrangulation on $\Sigma_g$ has fewer vertices. The number of vertices in a minimal quadrangulation on $\Sigma_g$ is denoted by ${|V|_{\min}(g)}$. It is surprising that no comprehensive formula is available thus far for ${|V|_{\min}(g)}$.

Current graphs and rotation schemes were used by Hartsfield and Ringel \cite {hr} to construct quadrangulations by the general octahedral graph $K_{p(2)}$ which is the complement of a 1-factor in the complete graph $K_{2p}$ ($p\ge 3$). These are all minimal (except for $K_{3(2)}$ on $\Sigma_1$), so that
\begin{equation} 
{|V|_{\min}(g)}={3+\sqrt{8g+1}}
\end{equation}
whenever $g={\frac 12 (p-1)(p-2)}$ with $p \ge 4$, in which case the square root is always an integer and, moreover, the value on the right hand side is an even integer ($=2p$). In this note, by using a different technique, called the {\it spinal method}, Hartsfield and Ringel's Formula (1) is extended as follows.

\begin{theorem}
For $g \ge 3$,
\begin{equation}
{|V|_{\min}(g)} = 2 \Big {\lceil} \frac {3+\sqrt{8g+1}} 2 \Big {\rceil}
\end{equation}
whenever
$$\Big {\lceil} \frac {3+\sqrt{8g+1}} 2 \Big {\rceil} = \Big {\lfloor}\frac {7+\sqrt{32g-15}} 4 \Big {\rfloor}.$$  
\end{theorem}

Theorem 1.1 will be proved at the end of this note. In Section 3 we'll see that there are infinitely many values of $g$ to which Formula (2) applies but Formula (1) doesn't. For example, Formula (2) gives ${|V|_{\min}(53)} = 24$ whereas Formula (1) does not apply to $g=53$. 

\section{Spinal quadrangulations}

The {\it (first) Betti number} of a connected graph $G$ is given by $\beta (G)=|E(G)|-|V(G)|+1$, where $|V(G)|$ and $|E(G)|$ stand for the cardinalities of the vertex and edge sets of $G$ (respectively). The cardinality $|V(G)|$ is also called the {\it order} of $G$. Especially, the complete graph $K_p$ of order $p$ has $\beta(K_p)=\frac 12 (p-1)(p-2)$. 

The {\it 2-fold interlacement} of $G$ is denoted by ${G [:]}$ and is defined to be the graph which has vertex set ${V(G[:])}=V(G') \cup V(G'')$, where  $G'$ and $G''$ are two disjoint copies of $G$, and as edges the set ${E(G[:])}=E(G') \cup E(G'')$ plus the edges that join each vertex $v' \in V(G')$ to each vertex in $V(G'')$ that is adjacent (in $G''$) to the corresponding vertex $v'' \in V(G'')$. For instance, $K_p[:] = K_{p(2)}$ ($p\ge 2$).

\begin{theorem} 
{\bf (White \cite {w}, Craft \cite {c})}
For any non-trivial connected graph $G$, there exists a quadrangulation ${G[:]} \hookrightarrow \Sigma_{\beta(G)}$.
\end{theorem}

Any quadrangulation ${G[:]} \hookrightarrow \Sigma_{\beta(G)}$ is called a {\it spinal quadrangulation} with {\it spine} $G$ and {\it genus} $g=\beta(G)$. By Theorem 2.1, the genus of a spinal quadrangulation is equal to the Betti number of the spine. 

\section{Minimal quadrangulations}
\label{2}
We start up  with two corollaries of Theorem 2.1.

\begin{corollary}
For any integer $p \ge 2$ there exists a quadrangulation ${K_p[:]} \hookrightarrow \Sigma_g$ with $g = \beta(K_p)=\frac 1 2 (p-1)(p-2)$.
\end{corollary}

\begin{corollary}
Let $p$ be an integer $\ge3$ and let $(K_p-e)$ be any graph formed by deleting an edge from $K_p$. Then there exists a quadrangulation ${(K_p-e)[:]} \hookrightarrow \Sigma_g$ with $g=\beta(K_p)-1$.
\end{corollary}

The quadrangulations of Corollaries 3.1 and 3.2 were first discovered (in terms of general octahedral graphs) by Hartsfield and Ringel \cite{hr}, who also showed their minimality for $\Sigma_g$ whenever $p \ge 4$ for Corollary 3.1, and $p \ge 8$ for Corollary 3.2. These corollaries are special cases of Lemma 3.3, stated shortly, and correspond to the particular cases $m=0$ and $m=1$, respectively. Using the method of current graphs for $m \ge 2$ would have been very complicated, so Hartsfield and Ringel had to stop at $m=1$. In contrast, the spinal method enables generalization of the results of \cite{hr} to an arbitrary $m$ not exceeding $\frac 14 p-1$. This demonstrates one way in which the spinal method is useful.

\begin{lemma}
Let $(K_p-me)$ be any graph formed by deleting $m$ edges from $K_p$ ($0 \le m \le \beta(K_p)$). If $(K_p-me)$ is connected, there exists a quadrangulation ${(K_p-me) [:]} \hookrightarrow \Sigma_g$ with $g = \beta(K_p)-m$. Moreover, any such quadrangulation is minimal for $\Sigma_g$ with $g \ge 1$ whenever $p \ge 4(m+1)$.
\end{lemma}

\begin{proof}
Compute $g=\beta(K_p-me) = \frac12 (p-1)(p-2)-m$, and the existence follows from Theorem 2.1. It remains to prove minimality of the quadrangulation constructed. Let $\alpha_0$, $\alpha_1$, $\alpha_2$ denote the number of vertices, edges, and regions of an arbitrary quadrangulation on $\Sigma_g$. By Euler's equation, we have $\alpha_0 - \alpha_1 + \alpha_2 = 2-2g$. Furthermore, since any pair of vertices are joined by at most one edge, we have $\alpha_1 \le {\alpha_0\choose 2}$, and since each edge meets exactly two regions, we have $4\alpha_2 = 2\alpha_1$. From these it can be derived that $\alpha_0^2 –-5\alpha_1 +(8-8g) \ge 0$. This quadratic inequality has the solution:
\begin{equation}
\alpha_0 \ge \Big\lceil \frac 12 \Big{(}5+\sqrt{32g-7} \Big{)} \Big\rceil.
\end{equation}
With $g$ computed in the beginning of the proof, we find
\begin{equation}
32g-7 = 16p^2 -48p+25-32m. 
\end{equation}
Now, since the constructed quadrangulation has $2p$ vertices, it follows from Eqs. (3) and (4) that it is minimal for $\Sigma_g$ whenever the following double inequality holds:   
$$
2p-1 < \frac12 \Big{(} 5+\sqrt{16p^2-48p+25-32m} \Big{)} \le 2p.
$$
This can be rewritten as $16p^2-56p+49<16p^2-48p+25-32m\le16p^2-40p+25$, or as
$$
\left\{     \begin{array}{lr}       8p>24+32m,\\      8p \ge -32m.    \end{array}   \right.
$$
Since the second inequality is a tautology, we conclude that $p>3+4m$, hence $p\ge4+4m$.         
The proof is complete. 
\end{proof}

Lemma 3.3 provides infinitely many new minimal quadrangulations (for infinitely many values of $g$) not covered by Hartsfield and Ringel \cite{hr}. The new minimal quadrangulations correspond to the values of $m$ satisfying the double inequality: $2 \le m \le \frac 14 p -1$. For example, the quadrangulations ${K_{12}[:]} \hookrightarrow \Sigma_{55}$, ${(K_{12}-e)[:]} \hookrightarrow \Sigma_{54}$, and ${(K_{12}-2e)[:]} \hookrightarrow \Sigma_{53}$ are all minimal for the corresponding 2-manifolds (and have 132, 130, and 128 regions, respectively). The first two are covered by \cite{hr} (or Corollaries 3.1 and 3.2), but the minimal quadrangulation on $\Sigma_{53}$ is a new one. Note that a quadrangulation is minimal in the sense of \cite{hr} if it has the minimum number of regions, but since the definition assumes the 2-manifold is fixed, their definition of a minimal quadrangulation agrees with the one given in the beginning of this note.

The construction using spines creates quadrangulations of an easily controlled order. 

\begin{corollary}
For any integers $g \ge 0$ and $p \ge 2$ such that $g \le \beta(K_p)$, there exists a spinal quadrangulation on $\Sigma_g$ with order $2p$.
\end{corollary}
\begin{proof}
The order of a spinal quadrangulation is twice the order of its spine. By Lemma 3.3, we can take $(K_p-me)$ as a spine, letting $m$ be $(\beta(K_p)-g)$. Clearly, it is possible to remove this number of edges from $K_p$ so that the remaining graph is still connected.
\end{proof}

The spectrum $\{2p\}$ described in Corollary 3.4 for the possible orders is full for fixed genus; there are infinitely many spinal quadrangulations of arbitrarily large even order in each 2-manifold $\Sigma_g$. Solving the quadratic equation $\beta(K_p)=g$ for $p$, we come to the following comprehensive formula for the order of a minimal spinal quadrangulation with genus $g \ge 0$: 
\begin{equation}
{|V|_{\min}^{{\rm{sp}}}(g)} = 2 \Big{\lceil} \frac {3+\sqrt{8g+1}} 2 \Big{\rceil}
\end{equation}
and therefore
$$
{|V|_{\min}(g)} \le 2 \Big{\lceil} \frac {3+\sqrt{8g+1}} 2 \Big{\rceil}.
$$

For example, ${|V|_{\min}^{{\rm{sp}}}(0)} ={|V|_{\min}(0)} =4$ by Eq.~(5). It should be noted that Hartsfield and Ringel \cite{hr} assert that ${|V|_{\min}(0)} =8$ because their definition of a quadrangulation, in comparison to the definition given in the beginning of this note, has an additional requirement as follows: the intersection of any two distinct regions is either empty or at most one edge and at most three vertices. However, any minimal quadrangulation in the sense of our definition has no vertices of degree 2 whenever $g\ge1$, and therefore satisfies this requirement. By Eq.~(5), ${|V|_{\min}^{{\rm{sp}}}(1)} =6$ and ${|V|_{\min}^{{\rm{sp}}}(2)} =8$, whereas it is shown in \cite{hr} that ${|V|_{\min}(1)} =5$ and ${|V|_{\min}(2)} =7$.

We now prove Theorem 1.1 stated in the Introduction.

\begin{proof}{\it (for Theorem 1.1)}
Denote $\frac12 (3+\sqrt{8g+1})$ by $a(g)$ and denote $\frac14 (7+\sqrt{32g-15})$ by $b(g)$. It follows from Lemma 3.3 that ${|V|_{\min}(g)} = {|V|_{\min}^{{\rm{sp}}}(g)}=2p$ whenever this double inequality holds: $\beta(K_p)+1- \frac 14 p \le g \le \beta(K_p)$ or, equivalently, $a(g) \le p \le b(g)$.
Note that whenever $g\ge3$ (and so $p \ge 4$) we have $a(g) \le b(g)$, and also observe that the closed interval $[a(g), b(g)]$ has length $<1$ and can contain at most one integer, $\ell(g)$. Furthermore, such an $\ell(g)$ exists if $\lceil a(g) \rceil = \lfloor b(g)\rfloor$ ($=\ell(g)$), in which case we have ${|V|_{\min}(g)} = {|V|_{\min}^{{\rm{sp}}}(g)} =2\ell(g)$. 
\end{proof}











\bibliographystyle{model1-num-names}
\bibliography{<your-bib-database>}

\begin{thebibliography}{00}


\bibitem{c}
D.L. Craft, On the genus of joins and compositions of graphs, Discrete Math. 178 (1998) 25--50.
\bibitem{hr}
N. Hartsfield, G. Ringel, Minimal quadrangulations of orientable surfaces, J. Comb. Theory, Ser. B 46 (1989) 84--95.
\bibitem{w}
A.T. White, On the genus of the composition of two graphs, Pacific J. Math. 41 (1972) 275--279.
 
\end{thebibliography}

\end{document}